\documentclass[preprint,12pt]{elsarticle}
\usepackage{amssymb}
\usepackage{amsthm}
\usepackage{amsmath}
\usepackage{epic}
\newtheorem{thm}{Theorem}[section]
\newtheorem{cor}[thm]{Corollary}

\newtheorem{defi}[thm]{Definition}

\newtheorem{example}[thm]{Example}
\newtheorem{pro}[thm]{Proposition}
\journal{arXiv}

\begin{document}

\begin{frontmatter}

\title{Some results on the generalized inverse of tensors and idempotent tensors}
\author[label3]{Lizhu Sun}\ead{sunlizhu678876@126.com}
\author[label3]{Baodong Zheng}\ead{zbd@hit.edu.cn}
\author[label1]{Changjiang Bu}\ead{buchangjiang@hrbeu.edu.cn}
\author[label2]{Yimin Wei}\ead{ymwei@fudan.edu.cn}

\address[label3]{School of Science, Harbin Institute of Technology, Harbin 150001, PR China}
\address[label1]{College of Science, Harbin Engineering University, Harbin 150001, PR China}

\address[label2]{School of Mathematical Sciences and Shanghai Key Laboratory of Contemporary Applied Mathematics, Fudan University, Shanghai, 200433, PR China}

\begin{abstract}
Let $\mathcal{A}$ be an order $t$ dimension $m\times n\times \cdots \times n$ tensor over complex field. In this paper, we study some {generalized inverses} of $\mathcal{A}$, the {$k$-T-idempotent tensors} and the idempotent tensors based on the general tensor product. Using the tensor generalized inverse, some solutions of the equation $\mathcal{A}\cdot x^{t-1}=b$ are given, where $x$ and $b$ are dimension $n$ and $m$ vectors, respectively. The {generalized inverses} of some block tensors, the eigenvalues of {$k$-T-idempotent tensors} and idempotent tensors are given. And the relation between the generalized inverses of tensors and the $k$-T-idempotent tensors is also showed.
\end{abstract}


\end{frontmatter}

\section{Introduction}
\label{Introduction }

In recent years, there has been extensive attention and interest in the work of spectral theory of tensors and hypergraphs \cite{wei 2013}-\cite{Qi Shao}, since the research of Qi \cite{Qi2005}, Lim \cite{Lim} and Shao \cite{shao 2013}.

For a positive integer $n$, let $[n]=\{1,\ldots,n\}$. An order $t$ tensor $\mathcal{A}=(a_{i_1\cdots i_t})_{1\leq i_j\leq n_j}$ ($j=1,\ldots, t$)
is a multidimensional array with $ n_1n_2\cdots n_t$ entries.
Let $\mathbb{C}^{n_1\times\cdots\times n_t}$ be the set of all the order $t$ dimension $ n_1\times\cdots\times n_t$ tensors over complex field $\mathbb{C}$. Clearly, an order $2$ tensor is a matrix.
Let $\mathbb{C}_{t}^{m,n}$ denote the set of all the $m\times n\times \cdots \times n$ tensors of order $t$ over complex field. Let $\mathcal{D}={\rm diag}(d_{1\cdots 1},\ldots,d_{n\cdots n})\in\mathbb{C}_{t}^{n,n}$ be a diagonal tensor whose entries are all zero except for $d_{ii\cdots i}$, $i=1, \ldots, n$.
If $d_{ii\cdots i}=1$, $i=1,\ldots, n$, then $\mathcal{D}$ is the unit tensor with order $t$, denoted by $\mathcal{I}$.
For the vector $x=(x_1,x_2,\ldots ,x_n)^\mathrm{T}\in \mathbb{C}^n$ and the tensor $ \mathcal{A}\in\mathbb{C}_{t}^{m,n}$, let $\mathcal{A}\cdot x^{t-1}$ be a dimension $m$ vector whose $i$-th component is
\begin{eqnarray}\label{Eigen}
(\mathcal{A}\cdot x^{t-1})_i=\sum\limits_{{i_2}, \ldots ,{i_t} \in [n]} {a_{ii_2\cdots i_t}x_{i_2}x_{i_3}\cdots x_{i_t}},
\end{eqnarray}
where $i\in[m]$. In \cite{Qi2005}, Qi defined the \textit{eigenvalue of tensors}. For $ \mathcal{A}\in\mathbb{C}_{t}^{n,n}$, $\lambda\in\mathbb{C}$ is called the eigenvalue of $\mathcal{A}$ if there exists nonzero vector $x\in\mathbb{C}^n$ such that
$
\mathcal{A}\cdot x^{t-1}=\lambda x^{[t-1]}
$,
where $x^{[r]}=(x_1^r,x_2^r,\ldots ,x_n^r)^\mathrm{T}$.

Shao defined the \textit{general tensor product} in \cite{shao 2013}. For tensors $\mathcal{A}\in \mathbb{C}_{t}^{n,n}$ and $\mathcal{B}\in \mathbb{C}_{k}^{n,n}$, the {general tensor product} of them is  $\mathcal{A}\mathcal{B}\in \mathbb{C}_{(t-1)(k-1)+1}^{n,n}$ with entries as
\begin{eqnarray}\label{product}
(\mathcal{A}\mathcal{B})_{i\alpha_1\cdots \alpha_{t-1}}=\sum_{i_2,\ldots,i_t\in[n]}a_{ii_2\cdots i_t}b_{i_2\alpha_1}\cdots b_{i_t\alpha_{t-1}},
\end{eqnarray}
where $i\in[n]$, $\alpha_1,\ldots,\alpha_{t-1}\in[n]^{k-1}$.
Bu et al. showed that the {general product} of $\mathcal{A}\in \mathbb{C}_{t}^{m,n}$ and $\mathcal{B}\in \mathbb{C}_{k}^{n,m}$ also can be written as Eq.(\ref{product}), where $i\in[m]$, $\alpha_1,\ldots,\alpha_{t-1}\in[m]^{k-1}$.

The definition of the \textit{inverse of tensors} was given in \cite{Bu 2014}. For $\mathcal{A}\in \mathbb{C}_{t}^{n,n}$,  $\mathcal{B}\in \mathbb{C}_{k}^{n,n}$, if $\mathcal{A}\mathcal{B}=\mathcal{I}$, then $\mathcal{A}$ is called an order $t$ left inverse of $\mathcal{B}$, denoted by $\mathcal{B}^{L_{t}}$; $\mathcal{B}$ is called an order $k$ right inverse of $\mathcal{A}$, denoted by $\mathcal{A}^{R_{k}}$.

It is well known that there are many types of {generalized inverses} of matrices (operators) \cite{GI}.
Let $\mathbb{H}$ be a Hilbert space and $\mathfrak{L}(\mathbb{H})$ be the set of the linear operators on $\mathbb{H}$.
 Let $A\in \mathbb{C}^{m\times n}~(\in \mathfrak{L}(\mathbb{H}))$, $X\in \mathbb{C}^{n\times m}~(\in \mathfrak{L}(\mathbb{H}))$,
\[
{\rm (1)}~AXA=A;~~{\rm (2)}~XAX=X;~~{\rm (5)}~XA=AX.
\]
The matrix (operator) $X$ is said to be the \textit{$\{i\}$ inverse} of the matrix (operator) $A$ if the above equation (i) holds.
$A\{i\}$ is the set of all the {\{i\} inverses} of $A$. The matrix (operator) $X$ is called the \textit{group inverse} of the matrix (operator) $A$ if the equations (1), (2) and (5) hold, denoted by $A^\#$. And the group inverse is a kind of spectral generalized inverse \cite{GI}.


The generalized inverses of block matrices (operators) have important applications in algebraic connectivity and algebraic bipartiteness of graphs \cite{G 1}, Markov chains ([14,15]), resistance distance \cite{G 5} and so on.
Scholars gave many results on the representations of the generalized inverses of block matrices \cite{G 4}-\cite{G 7}.

In this paper, we study the generalized inverses of tensors, the $k$-T-idempotent tensors and the idempotent tensors.
This paper is organized as follows. In section 2, we show the definitions of some {generalized inverses} of tensors.
And using the generalized inverse, some solutions of the equation $\mathcal{A}\cdot x^{t-1}=b$ is gotten, where $\mathcal{A}\in\mathbb{C}_{t}^{m,n}$, the vectors $x\in\mathbb{C}^{n}$, $b\in\mathbb{C}^{m}$.
The tensor generalized inverses of some block tensors are given.
In the section 3, we give the definitions of the $k$-T-idempotent tensors and idempotent. There also some results on the eigenvalues of the $k$-T-idempotent tensors and idempotent tensors; the relation between the $k$-T-idempotent tensors (idempotent tensors) and the tensor generalized inverses.
In the section 4, some examples of tensor generalized inverses and idempotent tensors are presented.

\section{Generalized inverse of tensors}

In this section, we show the \{1\} inverse of a tensor $ \mathcal{A}\in\mathbb{C}_{t}^{m,n}$ first.

For a tensor $\mathcal{A}\in\mathbb{C}_{t}^{m,n}$, let $\textbf{R}(\mathcal{A})=\{\mathcal{A}\cdot y^{t-1} | ~y\in\mathbb{C}^n\}$, $\textbf{N}(\mathcal{A})=\{x |~ \mathcal{A}\cdot x^{t-1}=0,~x\in \mathbb{C}^n\}$.
Obviously, the equation
$\mathcal{A}\cdot x^{t-1}=b$ is solvable iff $b\in\textbf{R}(\mathcal{A})$,
where $x\in\mathbb{C}^n$. Next, we consider the problem that
for the
 tensor $ \mathcal{A}\in\mathbb{C}_{t}^{m,n}$,
whether  $\mathcal{G}\cdot (b^{[\frac{1}{s}]})^{k-1}$ is a solution of $\mathcal{A}\cdot x^{t-1}=b$ for all the vectors $b\in\textbf{R}(\mathcal{A})$ , where $ \mathcal{G}\in\mathbb{C}_{k}^{n,m}$ and $s=(t-1)(k-1)$. If it is a solution of $\mathcal{A}\cdot x^{t-1}=b$ for all the $b\in\textbf{R}(\mathcal{A})$. Then,
$ \mathcal{A}\mathcal{G}\cdot (b^{[\frac{1}{s}]})^s= b$, that is
$$ \mathcal{A}\mathcal{G}\cdot[(\mathcal{A}\cdot y^{t-1})^{[\frac{1}{s}]}]^{s}= \mathcal{A}\cdot y^{t-1},~for~all~y\in\mathbb{C}^n.$$
On the other hand, if $\mathcal{G}$ satisfies $ \mathcal{A}\mathcal{G}\cdot[(\mathcal{A}\cdot y^{t-1})^{[\frac{1}{s}]}]^{s}= \mathcal{A}\cdot y^{t-1}$ for all $y\in\mathbb{C}^n$.
 Since $b\in\textbf{R}(\mathcal{A})$, there exists a vector $y\in\mathbb{C}^n$ such that $  \mathcal{A}\mathcal{G}\cdot (b^{[\frac{1}{s}]})^s=  \mathcal{A}\mathcal{G}\cdot[(\mathcal{A}\cdot y^{t-1})^{[\frac{1}{s}]}]^{s}= \mathcal{A}\cdot y^{t-1}=b$. It implies that $ \mathcal{G}\cdot (b^{[\frac{1}{s}]})^{k-1}$ is a solution of $\mathcal{A}\cdot x^{t-1}=b$.

From the above discussion, we give the concept of the \{1\} inverse of tensors as follows.

\begin{defi}\label{defi [1]}
Let $ \mathcal{A}\in\mathbb{C}_{t}^{m,n}$ and $ \mathcal{X}\in\mathbb{C}_{k}^{n,m}$. If the equation
$$ \mathcal{A} \mathcal{X}\cdot [(\mathcal{A}\cdot y^{t-1})^{[\frac{1}{s}]}]^s= \mathcal{A}\cdot y^{t-1}$$
holds for all $y\in\mathbb{C}^n$, where $s=(t-1)(k-1)$, then $\mathcal{X}$ is called the order $k$ {\rm\{1\}} inverse of $ \mathcal{A}$, denoted by $\mathcal{A}^{(1)_{k}}$. Denote the set of all the order $k$ {\rm\{1\}} inverses of $\mathcal{A}$ by $\mathcal{A}\{1\}_{k}$.
\end{defi}
If an order $k$ right inverse of the tensor $\mathcal{A}\in\mathbb{C}_{t}^{n,n}$ exists, then $\mathcal{AA}^{R_k}=\mathcal{I}$, $\mathcal{AA}^{R_k}\cdot[(\mathcal{A}\cdot y^{t-1})^{[\frac{1}{s}]}]^s=\mathcal{I}\cdot[(\mathcal{A}\cdot y^{t-1})^{[\frac{1}{s}]}]^s=\mathcal{A}\cdot y^{t-1}$, where $y\in\mathbb{C}^n$ is an arbitrary vector and $s=(t-1)(k-1)$. Hence,  $\mathcal{A}^{R_k}$ is a \{1\} inverse of $\mathcal{A}$. The {\{1\} inverse} (with fixed order) of a tensor is not unique in general. If an order $k$ tensor $\mathcal{X}$ is a {\{1\} inverse} of $\mathcal{A}$, we write $\mathcal{A}^{(1)_k}=\mathcal{X}$.
When $t=k=2$, it is easy to see that Definition \ref{defi [1]} is the definition of the {\{1\} inverse} of matrices (see \cite{GI}).

\begin{pro}\label{pro 1}
Let $ \mathcal{A}\in\mathbb{C}_{t}^{m,n}$ and $x\in\mathbb{C}^n$. Let $\mathcal{A}^{(1)_k}$ denote a {\rm \{1\}} inverse of $\mathcal{A}$. If the equation $\mathcal{A}\cdot x^{t-1}=b$ is solvable, then $x=\mathcal{A}^{(1)_k}\cdot (b^{[\frac{1}{s}]})^{k-1}$ is a solution of it, where $s=(t-1)(k-1)$.
\end{pro}

For $a\in\mathbb{C}$, let
$
a^ +   = \left\{ {\begin{array}{*{20}c}
   {a^{ - 1} ,~a \ne 0,}  \\
   {0~,~a = 0.}  \\
\end{array}} \right.
$
\begin{pro}\label{unit 1th}
Let $\mathcal{A}={\rm diag}(a_1,a_2,\dots,a_n)\in\mathbb{C}_{t}^{n,n}$ be a diagonal tensor.
Then
\[{\rm diag}((a_1^+)^{\frac{1}{t-1}},(a_2^+)^{\frac{1}{t-1}}, \dots,(a_n^+)^{\frac{1}{t-1}})\in\mathbb{C}_{k}^{n,n}\]
is an order $k$ {\rm \{1\}} inverse of $\mathcal{A}$.
\end{pro}
\begin{proof}
Let $\mathcal{X}={\rm diag}((a_1^+)^{\frac{1}{t-1}},(a_2^+)^{\frac{1}{t-1}},\dots,(a_n^+)^{\frac{1}{t-1}})$.
It follows from Eq.(\ref{Eigen}) that $(\mathcal{A}\cdot y^{t-1})_i=a_iy_i^{t-1}$, $i\in[n]$, $y\in\mathbb{C}^n$. By Eq.(\ref{product}), it yields that $\mathcal{A}\mathcal{X}$ is a diagonal tensor with the diagonal entries $1$ or $0$. By directly computation, we have
$$ \mathcal{A} \mathcal{X}\cdot [(\mathcal{A}\cdot y^{t-1})^{[\frac{1}{s}]}]^s= \mathcal{A}\cdot y^{t-1},$$
for all the $y\in\mathbb{C}^n$, where $s=(k-1)(t-1)$. Hence, $\mathcal{X}$ is an order $k$ {\{1\} inverse} of $\mathcal{A}$.
\end{proof}

The order $k$ right inverse of unit tensor exists and is not unique in general (see \cite{Bu 2014}), the order $k$ right inverse of a unit tensor is the order $k$ \{1\} inverse of itself.


Next, we show some results on the $\{1\}$ inverse of ※block tensors§.
Let $\mathcal{A}=(a_{i_1\cdots i_t})\in\mathbb{C}_{t}^{m,n}$. Let $\widetilde{\mathcal{A}}=(a_{i_1\cdots i_t})\in\mathbb{C}_{t}^{r,l}$ be a subtensor of $\mathcal{A}$, where $r\leq m$, $l \leq n$ (see \cite{block}).

\begin{thm}\label{T 1}
Let $\mathcal{A}\in\mathbb{C}_{t}^{m,n}$ and let $\widetilde{\mathcal{A}}\in\mathbb{C}_{t}^{r,l}$ be the subtensor of $\mathcal{A}$. If the entries of $\mathcal{A}$ are all zero except for $\widetilde{\mathcal{A}}$. Then each tensor $\mathcal{G}\in \mathcal{A}\{1\}_k$ is a tensor with the subtensor $\widetilde{\mathcal{G}}\in \mathcal{A}_1\{1\}_k$ and all the other entries are arbitrary.

\end{thm}
\begin{proof}
Let $\mathcal{G}=(g_{i_1\cdots i_k})\in\mathbb{C}_{k}^{n,m}$ be an order $k$ {$\{1\}$ inverse} of $\mathcal{A}$. $\widetilde{\mathcal{G}}\in\mathbb{C}_{k}^{l,r}$ denotes the subtensor of $\mathcal{G}$. Let $y=\left(
                                                                                                           \begin{array}{c}
                                                                                                             Y_1 \\
                                                                                                             Y_2 \\
                                                                                                           \end{array}
                                                                                                         \right)
\in\mathbb{C}^n$ be an arbitrary vector, where $Y_1=(y_1, \ldots, y_l)^\mathrm{T}$ and $Y_2=(y_{l+1}, \ldots, y_n)^\mathrm{T}$.

By the general tensor product, it yields that the $i$-th component of vector $\mathcal{A}\cdot y^{t-1}$ is
$$
(\mathcal{A}\cdot y^{t-1})_i = \left\{ {\begin{array}{*{20}{c}}
  \sum\limits_{{i_2}, \ldots, {i_t} \in [l]} {a_{ii_2\cdots i_t}y_{i_2}\cdots y_{i_t}}= (\widetilde{\mathcal{A}}\cdot Y_1^{t-1})_i,~if ~i\leq r;\\
  0, ~if ~i>r.~~~~~~~~~~~~~~~~~~~~~~~~~~~~~~~~\\
\end{array}} \right.
$$
That is
$\mathcal{A}\cdot y^{t-1}=\left(
  \begin{array}{c}
    {\widetilde{\mathcal{A}}\cdot Y_1^{t-1} } \\
    0 \\
  \end{array}
\right)
 $.
By computing, we get
 $$(\mathcal{AG})_{i\alpha_1\cdots \alpha_{t-1}}=\sum\limits_{i_2,\ldots,i_t\in[l]}a_{ii_2\cdots i_t}g_{i_2\alpha_1}\cdots g_{i_t\alpha_{t-1}}
 =(\widetilde{\mathcal{A}}\widetilde{\mathcal{G}})_{i\alpha_1\cdots \alpha_{t-1}},$$
 if all the indices  in $i$, $\alpha_1,\dots, \alpha_{t-1}$ are less than or equal to $r$; $(\mathcal{AG})_{i\alpha_1\cdots \alpha_{t-1}}=0$ if $i>r$.

Let $z=(\mathcal{A}\cdot y^{t-1})^{[\frac{1}{s}]}
=
\left( {\begin{array}{*{20}c}
   {(\widetilde{\mathcal{A}}\cdot Y_1^{t-1})^{[\frac{1}{s}]} }  \\
   0  \\
\end{array}} \right)
$, where $s=(t-1)(k-1)$.
By Eq.(\ref{Eigen}), it yields that
\[(\mathcal{AG}\cdot z^s)_i= \left\{ {\begin{array}{*{20}{c}}
   (\widetilde{\mathcal{A}}\widetilde{\mathcal{G}}\cdot [(\widetilde{\mathcal{A}}\cdot Y_1^{t-1})^{[\frac{1}{s}]}]^s)_i,~if~i\leq r;  \\
  0,~if~i>r.~~~~~~~~~~~~~~~~ \\
\end{array}} \right.\]
That is $\mathcal{A}\mathcal{G}\cdot z^s=
 \left( {\begin{array}{*{20}c}
   {\widetilde{\mathcal{A}}\widetilde{\mathcal{G}}\cdot[(\widetilde{\mathcal{A}}\cdot Y_1^{t-1})^{[\frac{1}{s}]}]^s}  \\
   0  \\
\end{array}} \right)
$.
Since $\mathcal{G}$ is an order $k$ {\{1\} inverse} of $\mathcal{A}$, it yields that $\mathcal{A}\mathcal{G}\cdot z^s=\mathcal{A}\cdot y^{t-1}$,
 so $\widetilde{\mathcal{A}}\widetilde{\mathcal{G}}\cdot [(\widetilde{\mathcal{A}}\cdot Y_1^{t-1})^{[\frac{1}{s}]}]^s =\widetilde{\mathcal{A}}\cdot Y_1^{t-1} $.
Thus, we get $\widetilde{\mathcal{G}}$ is an order $k$ {$\{1\}$ inverse} of $\widetilde{\mathcal{A}}$ and all the other entries of $\mathcal{G}$ are arbitrary.
\end{proof}
 When the tensor $\mathcal{A}$ in Theorem \ref{T 1} is an order $2$ tensor, the following result can be gotten.
\begin{cor}
Let the block matrix $A=\left( {\begin{array}{*{20}c}
   {A_1 } & 0  \\
   0 & 0  \\
\end{array}} \right) \in \mathbb{C}^{m\times n}$ and $A_1 \in \mathbb{C}^{r\times l}$. Then
\[
A\{ 1\}  = \left\{ {\left. {\left( {\begin{array}{*{20}{c}}
   {W} & X  \\
   Y & Z  \\
\end{array}} \right)}\in \mathbb{C}^{n\times m} \right|} \right.\left. \begin{array}{l}
 W \in A_1\{ 1\}, ~
 X,~Y~and~Z~are~proper~\\matrices~with~arbitrary~entries \\
 \end{array} \right\}.\]

\end{cor}

The tensor $\mathcal{A}\in\mathbb{C}_{t}^{n,n}$ is a diagonal block tensor as
\begin{eqnarray}\label{block 4}
\mathcal{A}={\rm diag}(\mathcal{A}_1, \mathcal{A}_2),
\end{eqnarray}
where $\mathcal{A}_1=(a_{i_1\cdots i_t})$ ($i_1,\ldots, i_t\leq  r\leq n$, $r$ is a positive integer);
$\mathcal{A}_2=(a_{i_1\cdots i_t})$ ($i_1,\ldots, i_t>r$). And the other entries of $\mathcal{A}$ are all zero (see \cite{Hu}).

\begin{thm}\label{diag 1}
Let $\mathcal{A}$ be the form as in {\rm(\ref {block 4})}. Then ${\rm diag}(\mathcal{A}_1^{(1)_k}, \mathcal{A}_2^{(1)_k})\in\mathbb{C}_{k}^{n,n}$ is an order $k$ {\rm \{1\}} inverse of $\mathcal{A}$.
\end{thm}
\begin{proof}
Let $\mathcal{G}_1=\mathcal{A}_1^{(1)_k}$, $\mathcal{G}_2=\mathcal{A}_2^{(1)_k}$ and $\mathcal{G}={\rm diag}(\mathcal{G}_1, \mathcal{G}_2)$. And $y=\left(
                                                                                                           \begin{array}{c}
                                                                                                             Y_1 \\
                                                                                                             Y_2 \\
                                                                                                           \end{array}
                                                                                                         \right)
\in\mathbb{C}^n$ is an arbitrary vector, where $Y_1=(y_1, \ldots, y_r)^\mathrm{T}$ and $Y_2=(y_{r+1}, \ldots, y_n)^\mathrm{T}$.

By calculating, it yields that
\[(\mathcal{A}\cdot y^{t-1})_i = \left\{ {\begin{array}{*{20}{c}}
   (\mathcal{A}_1\cdot Y_1^{t-1})_i,~if~i\leq r;  \\
   (\mathcal{A}_2\cdot Y_2^{t-1})_i,~if~i>r.  \\
\end{array}} \right.\]
That is $\mathcal{A}\cdot y^{t-1}=\left(
                        \begin{array}{c}
                          {\mathcal{A}_1\cdot Y_1^{t-1}} \\
                          {\mathcal{A}_2\cdot Y_2^{t-1}} \\
                        \end{array}
                      \right)
$.
It follows from Eq.(\ref{product}) that:
if all the indices  in $i$, $\alpha_1,\ldots, \alpha_{t-1}$ are less than or equal to $r$, then
$$(\mathcal{AG})_{i\alpha_1\cdots \alpha_{t-1}} =(\mathcal{A}_1\mathcal{G}_1)_{i\alpha_1\cdots \alpha_{t-1}};$$
if all the indices  in $i$, $\alpha_1,\ldots, \alpha_{t-1}$ are greater than $r$, then
$$(\mathcal{AG})_{i\alpha_1\cdots \alpha_{t-1}} =(\mathcal{A}_2\mathcal{G}_2)_{i\alpha_1\cdots \alpha_{t-1}};$$
the other entries of $\mathcal{AG}$ are all zero.
It implies that $\mathcal{AG}={\rm diag}(\mathcal{A}_1\mathcal{G}_1, \mathcal{A}_2\mathcal{G}_2)$.

Let $z_1=\mathcal{A}_1\cdot (Y_1^{[\frac{1}{s}]})^{t-1}$, $z_2=\mathcal{A}_2\cdot (Y_2^{[\frac{1}{s}]})^{t-1}$ and
$z=\left(
          \begin{array}{c}
            {z_1} \\
            z_2 \\
          \end{array}
        \right)
=(\mathcal{A}\cdot y^{t-1})^{[\frac{1}{s}]}$, where $s=(t-1)(k-1)$.
By Eq.(\ref{Eigen}), we have
\[(\mathcal{AG}\cdot z^s)_i = \left\{ {\begin{array}{*{20}{c}}
   (\mathcal{A}_1\mathcal{G}_1\cdot z_1^s)_i=(\mathcal{A}_1\mathcal{G}_1\cdot [(\mathcal{A}_1\cdot Y_1^{t-1})^{[\frac{1}{s}]}]^s)_i=(\mathcal{A}_1\cdot Y_1^{t-1})_i,~if~i\leq r; \\
   (\mathcal{A}_2\mathcal{G}_2\cdot z_2^s)_i=(\mathcal{A}_2\mathcal{G}_2\cdot [(\mathcal{A}_2\cdot Y_2^{t-1})^{[\frac{1}{s}]}]^s)_i=(\mathcal{A}_2\cdot Y_2^{t-1})_i,~if~i>r. \\
\end{array}} \right.\]
That is $\mathcal{AG}\cdot z^s=\left(
                         \begin{array}{c}
                          \mathcal{A}_1\cdot Y_1^{t-1}\\
                           \mathcal{A}_2\cdot Y_2^{t-1}\\
                         \end{array}
                       \right)
$.
Thus, we get $\mathcal{AG}\cdot z^s=\mathcal{A}\cdot y^{t-1}$, so $\mathcal{G}$ is an order $k$ \{1\} inverse of $\mathcal{A}$.
\end{proof}

We partition a tensor $\mathcal{A}\in\mathbb{C}_{t}^{m,n}$ into the ※row blocks§ as
\begin{eqnarray}\label{block 2}
\mathcal{A}=\left(
  \begin{array}{c}
    \mathcal{A}_1\\
     \mathcal{A}_2 \\
  \end{array}
\right),
\end{eqnarray}
where $\mathcal{A}_1=(a_{i_1\cdots i_t})$ ($i_1\leq r$); $\mathcal{A}_2=(a_{i_1\cdots i_t})$ ($i_1> r$, $r\leq m$, $r$ is a positive integer).
And we also can partition $\mathcal{A}$ into the ※column blocks§ as
\begin{eqnarray}\label{block 3}
\mathcal{A}=\left(
  \begin{array}{cc}
    \mathcal{A}_1 & \mathcal{A}_2 \\
  \end{array}
\right),
\end{eqnarray}
where
$\mathcal{A}_1=(a_{i_1\cdots i_t})$ ($i_2,\ldots,i_t \leq r\leq n$); $\mathcal{A}_2=(a_{i_1\cdots i_t})$ otherwise (see \cite{Bu 2014}).

\begin{thm}\label{T 2}
{\rm{(1)}} Let $\mathcal{A}$ be the form as in {\rm(\ref {block 2})} and ${\cal A}_1\{1\}_k$ be the set of all the order k {\rm \{1\}} inverses of $\mathcal{A}_1$. If $\mathcal{A}_2=0$, then
\[\mathcal{A}{\{ 1\} _k} = \left\{ {\left( {\begin{array}{*{20}{c}}
   \mathcal{W} & \mathcal{X}  \\
\end{array}} \right)
 \in \mathbb{C} _k^{n,m}\left| {\begin{array}{*{20}{c}}
   {\mathcal{W} \in {\mathcal{A}_1}{{\{ 1\} }_k},~\mathcal{X}~is~a~proper}  \\
   {tensor~with~arbitrary~entries}  \\
\end{array}} \right.} \right\};\]

{\rm{(2)}} Let $\mathcal{A}$ be the form as in {\rm(\ref {block 3})} and ${\cal A}_1\{1\}_k$ be the set of all the order k {\rm \{1\}} inverses of $\mathcal{A}_1$. If $\mathcal{A}_2=0$, then
\[
{\cal A}{\{ 1\} _k} = \left\{ {\left( {\begin{array}{*{20}{c}}
   {{\cal W}} \\
   {\cal Y}  \\
\end{array}} \right)\in \mathbb{C}_k^{n,m}} \right.\left. {\left| \begin{array}{l}
~ \mathcal{W}\in {\mathcal{A}_1}{\{ 1\} _k},~{\cal Y}~is~a~proper~tensor \\
~whose~entries~are~arbitrary \\
 \end{array} \right.} \right\}.\]

\end{thm}
\begin{proof}
(1) Let the ※column blocks§ tensor $\mathcal{G}=(\mathcal{G}_1, \mathcal{G}_2)\in \mathbb{C}_k^{n,m}$ be an order $k$ {$\{1\}$ inverse} of $\mathcal{A}$, where $\mathcal{G}_1\in \mathbb{C}_k^{n,r}$. And $y=\left(
                                                                                                           \begin{array}{c}
                                                                                                             Y_1 \\
                                                                                                             Y_2 \\
                                                                                                           \end{array}
                                                                                                         \right)
\in\mathbb{C}^n$ is an arbitrary vector, where $Y_1=(y_1, \ldots, y_r)^\mathrm{T}$ and $Y_2=(y_{r+1}, \ldots, y_n)^\mathrm{T}$.

By calculating, it yields that the $i$-th component of vector $\mathcal{A}\cdot y^{t-1}$ is
\[
(\mathcal{A}\cdot y^{t-1})_i = \left\{ {\begin{array}{*{20}{c}}
   \sum\limits_{{i_2}, \ldots, {i_t} \in [n]} {a_{ii_2\cdots i_t}y_{i_2}y_{i_3}\cdots y_{i_t}}= (\mathcal{A}_1\cdot y^{t-1})_i,~if~i\leq r; \\
   (\mathcal{A}\cdot y^{t-1})_i=0,~if~i>r.~~~~~~~~~~~~~~~~~~~~~~~~~~  \\
\end{array}} \right.
\]
That is $\mathcal{A}\cdot y^{t-1}=\left(
                        \begin{array}{c}
                          {\mathcal{A}_1\cdot y^{t-1}} \\
                          {0}\\
                        \end{array}
                      \right)$.
It follows from the general tensor product that
\[
(\mathcal{AG})_{i\alpha_1\cdots \alpha_{t-1}} = \left\{ {\begin{array}{*{20}{c}}
  \sum\limits_{i_2,\ldots,i_t\in[n]}a_{ii_2\cdots i_t}g_{i_2\alpha_1}\cdots g_{i_t\alpha_{t-1}}
 =(\mathcal{A}_1\mathcal{G})_{i\alpha_1\cdots \alpha_{t-1}},~if~i\leq r; \\
   0,~if~i>r.~~~~~~~~~~~~~~~~~~~~~~~~~~~~~~~~~~~~~~~~~~~~~~~~~  \\
\end{array}} \right.
\]
It implies that $\mathcal{AG}=\left(
                                \begin{array}{c}
                                  {\mathcal{A}_1\mathcal{G}} \\
                                  {0}\\
                                \end{array}
                              \right)$.
Let $z=(\mathcal{A}\cdot y^{t-1})^{[\frac{1}{s}]}=\left(
                        \begin{array}{c}
                          {(\mathcal{A}_1\cdot y^{t-1})^{[\frac{1}{s}]}} \\
                          {0}\\
                        \end{array}
                      \right)$,  where $s=(t-1)(k-1)$.
By calculating, we have
\[
(\mathcal{AG}\cdot z^s)_i = \left\{ {\begin{array}{*{20}{c}}
   (\mathcal{A}_1\mathcal{G}_1\cdot [(\mathcal{A}_1\cdot y^{t-1})^{[\frac{1}{s}]}]^s)_i,~if~i\leq r;  \\
   0,~if~i>r.~~~~~~~~~~~~~~~~~  \\
\end{array}} \right.
\]
Since $\mathcal{G}$ is a {$\{1\}$ inverse} of $\mathcal{A}$, that is $\mathcal{AG}\cdot z^s=\mathcal{A}\cdot y^{t-1}$,
 so $\mathcal{A}_1\mathcal{G}_1\cdot [(\mathcal{A}_1\cdot y^{t-1})^{[\frac{1}{s}]}]^s =\mathcal{A}_1\cdot y^{t-1} $.
Thus, we get $\mathcal{G}_1$ is a $\{1\}$ inverse of $\mathcal{A}_1$ and $\mathcal{G}_2$ is arbitrary.

(2) Let the ※row blocks§ tensor $\mathcal{G}=\left(
                       \begin{array}{c}
                         {\mathcal{G}_1}\\
                         {\mathcal{G}_2} \\
                       \end{array}
                     \right)\in \mathbb{C}_k^{n,m}$
                     be an order $k$ \textit{$\{1\}$ inverse} of $\mathcal{A}$, where $\mathcal{G}_1\in \mathbb{C}_k^{r,m}$.
And $y=\left(
                                                                                                           \begin{array}{c}
                                                                                                             Y_1 \\
                                                                                                             Y_2 \\
                                                                                                           \end{array}
                                                                                                         \right)
\in\mathbb{C}^n$ is an arbitrary vector, where $Y_1=(y_1, \ldots, y_r)^\mathrm{T}$ and $Y_2=(y_{r+1}, \ldots, y_n)^\mathrm{T}$.

By calculation, it yields that
\begin{eqnarray*}
(\mathcal{A}\cdot y^{t-1})_i&=&\sum\limits_{{i_2} \cdots {i_t} \in [n]^{t-1}} {a_{ii_2\cdots i_t}y_{i_2}y_{i_3}\cdots y_{i_t}}\\
&~&+\sum\limits_{{i_2} \cdots {i_t}  \notin [n]^{t-1}} {a_{ii_2\cdots i_t}y_{i_2}y_{i_3}\cdots y_{i_t}}\\
&=&(\mathcal{A}_1\cdot Y_1^{t-1}+0)_i.
\end{eqnarray*}

It follows from the general tensor product that
\begin{eqnarray*}
 (\mathcal{AG})_{i\alpha_1\cdots \alpha_{t-1}}&=&\sum\limits_{i_2 \cdots i_t\in[n]^{t-1}}a_{ii_2\cdots i_t}g_{i_2\alpha_1}\cdots g_{i_t\alpha_{t-1}}\\
 &~&+\sum\limits_{i_2 \cdots i_t\notin[n]^{t-1} }a_{ii_2\cdots i_t}g_{i_2\alpha_1}\cdots g_{i_t\alpha_{t-1}}\\
 &=&(\mathcal{A}_1\mathcal{G}_1+0)_{i\alpha_1\cdots \alpha_{t-1}}.
 \end{eqnarray*}

Let $z=(\mathcal{A}\cdot y^{t-1})^{[\frac{1}{s}]}=(\mathcal{A}_1\cdot Y_1^{t-1})^{[\frac{1}{s}]}$, where $s=(t-1)(k-1)$.
Then $(\mathcal{A}\mathcal{G}\cdot z^s)_i=(\mathcal{A}_1\mathcal{G}_1\cdot [(\mathcal{A}_1\cdot Y_1^{t-1})^{[\frac{1}{s}]}]^s)_i$.
Since $\mathcal{G}$ is a {\{1\} inverse} of $\mathcal{A}$, it yields that $\mathcal{A}\mathcal{G}\cdot z^s=\mathcal{A}\cdot y^{t-1}$,
so $\mathcal{A}_1\mathcal{G}_1\cdot [(\mathcal{A}_1\cdot Y_1^{t-1})^{[\frac{1}{s}]}]^s=\mathcal{A}_1\cdot Y_1^{t-1}$.
Thus, we get $\mathcal{G}_1$ is a $\{1\}$ inverse of $\mathcal{A}_1$ and $\mathcal{G}_2$ is arbitrary.
\end{proof}
When the tensor $\mathcal{A}$ in Theorem \ref{T 2} is an order $2$ tensor, the following result can be gotten.
\begin{cor}
{\rm{(1)}} Let the block matrix $
A = \left( {\begin{array}{*{20}c}
   {A_1 } & 0  \\
\end{array}} \right)\in\mathbb{C}^{m\times n}
$ and $A_1\in \mathbb{C}^{m\times r}$. Then
\[
A\{ 1\}  = \left\{ {\left( {\begin{array}{*{20}{c}}
   {W}  \\
   Y  \\
\end{array}} \right)\in \mathbb{C}^{n\times m}} \right.\left. {\left| \begin{array}{l}
 W \in A_1\{ 1\} , ~Y~is~a~proper\\
~matrix~with~arbitrary~entries \\
 \end{array} \right.} \right\};\]

{\rm{(2)}} Let the block matrix $
A = \left( {\begin{array}{*{20}c}
   {A_1 }  \\
   0  \\
\end{array}} \right)\in\mathbb{C}^{m\times n}
$ and $A_1\in \mathbb{C}^{r\times n}$. Then
\[A{\{ 1\} } = \left\{ {\left( {\begin{array}{*{20}{c}}
   W & X  \\
\end{array}} \right)
 \in \mathbb{C}^{n\times m}\left| {\begin{array}{*{20}{c}}
   {W \in {A_1}{{\{ 1\} }},~X~is~a~proper}  \\
   {matrix~with~arbitrary~entries}  \\
\end{array}} \right.} \right\}.\]

\end{cor}

\begin{thm}
Let $\mathcal{A}={\rm diag}(a_1,a_2,\ldots, a_n)\in \mathbb{C}_t^{n,n}$ be a diagonal tensor, where $a_i\neq 0$ $(i=1,\ldots,n)$. Then the order $2$ $\{1\}$ inverse of $\mathcal{A}$ is the following diagonal matrix
$${\rm diag}(a_1^{-\frac{1}{t-1}}, a_2^{-\frac{1}{t-1}}, \ldots, a_n^{-\frac{1}{t-1}})\in \mathbb{C}^{n\times n}.$$
\end{thm}
\begin{proof}
Let $X=(x_{ij})\in \mathbb{C}^{n\times n}$ be the order $2$ $\{1\}$ inverse of $\mathcal{A}$ and $y\in\mathbb{C}^n$ be an arbitrary vector.
By computation, we have the components of $\mathcal{A}\cdot y^{t-1}$ and $X(\mathcal{A}\cdot y^{t-1})^{[\frac{1}{t-1}]}$ are
\begin{eqnarray*}
(\mathcal{A}\cdot y^{t-1})_i=a_iy_i^{t-1}
~{\rm and}~
({X}(\mathcal{A}\cdot y^{t-1})^{[\frac{1}{t-1}]})_i= \sum\limits_{j=1}^{n}x_{ij}a_j^{\frac{1}{t-1}}y_jㄛ
\end{eqnarray*}
respectively. Then the component of $\mathcal{A} {X}\cdot [(\mathcal{A}\cdot y^{t-1})^{[\frac{1}{t-1}]}]^{t-1}$ is
\[
(\mathcal{A} {X}\cdot [(\mathcal{A}\cdot y^{t-1})^{[\frac{1}{t-1}]}]^{t-1})_i=a_i\left(\sum\limits_{j=1}^{n}x_{ij}a_j^{\frac{1}{t-1}}y_j\right)^{t-1}.
\]
It follows from the Definition of tensor \{1\} inverse that $ \mathcal{A} {X}\cdot [(\mathcal{A}\cdot y^{t-1})^{[\frac{1}{t-1}]}]^{t-1}= \mathcal{A}\cdot y^{t-1}$, that is
\[
a_i\left(\sum\limits_{j=1}^{n}x_{ij}a_j^{\frac{1}{t-1}}y_j\right)^{t-1}=a_iy_i^{t-1},~i=1,~\ldots,~n.
\]
Note that the above equation holds for all $y\in\mathbb{C}^n$, it is easy to see that $x_{ij}=0$ if $ i\neq j$ and $x_{ii}=a_i^{-\frac{1}{t-1}}$ ($i,~j=1, \ldots, n$).
\end{proof}

By the above theorem and Theorem \ref{T 1}, the following result can be gotten.
\begin{thm}
Let $\mathcal{A}\in\mathbb{C}_{t}^{m,n}$ and let $\widetilde{\mathcal{A}}={\rm diag}(a_1, a_2, \ldots, a_r)\in\mathbb{C}_{t}^{r,r}$ be the subtensor of $\mathcal{A}$, where $a_i\neq 0,~i=1,\ldots,r$ . If the entries of $\mathcal{A}$ are all zero except for $\widetilde{\mathcal{A}}$, then the set of the order $2$ $\{1\}$ inverse of $\mathcal{A}$ is
\[
\mathcal{A}\{1\}_2 = \left\{ {\left. {\left( {\begin{array}{*{20}{c}}
   X & Y  \\
   Z & W  \\
\end{array}} \right)}\in \mathbb{C}^{n\times m} \right|\begin{array}{*{20}{c}}
   {X ={\rm diag}(a_1^{ - \frac{1}{{t - 1}}}, \ldots a_r^{ - \frac{1}{{t - 1}}}),~Y,~Z~and~W~ are}  \\
   {proper~matrices~with~arbithary~entries~~~~}  \\
\end{array}} \right\}.
\]

\end{thm}

\begin{thm}\label{THM 2}
Let $\mathcal{A}\in\mathbb{C}_{t}^{m,n}$, $\mathcal{B}\in\mathbb{C}_{t}^{m, n_1}$ and let $\mathcal{B}^{(1)_k}\in\mathbb{C}_{k}^{n_1,m}$ be an
order $k$ $\{1\}$ inverse of $\mathcal{B}$.
If $\mathcal{A}=P\mathcal{B}Q$, then $\mathcal{A}^{(1)_k}=Q^{(1)}\mathcal{B}^{(1)_k}P^\mathrm{T}$, where $P\in\mathbb{C}^{m\times m}$ is a permutation matrix, $Q\in\mathbb{C}^{n_1\times n}$ is a matrix with full row rank.
\end{thm}
\begin{proof}
Since $Q$ is a full row rank matrix, then $QQ^{(1)}=I$, where $I$ is a unit matrix (see \cite{GI}).
Let $\mathcal{G}=Q^{(1)}\mathcal{B}^{(1)_k}P^\mathrm{T}$ and $y\in \mathbb{C}^n$ is an arbitrary vector, we have
$$(P\mathcal{B}Q\cdot y^{t-1})^{[\frac{1}{s}]}=P(\mathcal{B}Q\cdot y^{t-1})^{[\frac{1}{s}]},$$
where $s=(t-1)(k-1)$. By computation, it yields that
\begin{align}
\mathcal{AG}\cdot [(\mathcal{A}\cdot y^{t-1})^{[\frac{1}{s}]}]^s&=P\mathcal{B}QQ^{(1)}\mathcal{B}^{(1)_k}P^\mathrm{T}\cdot [(P\mathcal{B}Q\cdot y^{t-1})^{[\frac{1}{s}]}]^s\notag \\
&=P\mathcal{BB}^{(1)_k}\cdot [(\mathcal{B}Q\cdot y^{t-1})^{[\frac{1}{s}]}]^s.\notag
\end{align}
It follows from Definition \ref{defi [1]} that
$$P\mathcal{BB}^{(1)_k}\cdot [(\mathcal{B}Q\cdot y^{t-1})^{[\frac{1}{s}]}]^s=P\mathcal{B}Q\cdot y^{t-1}=\mathcal{A}\cdot y^{t-1}.$$
Thus, we get $\mathcal{AG}\cdot [(\mathcal{A}\cdot y^{t-1})^{[\frac{1}{s}]}]^s=\mathcal{A}\cdot y^{t-1}$, so $\mathcal{G}$ is an order $k$ {\{1\} inverse} of $\mathcal{A}$.
\end{proof}
\begin{cor}
Let $\mathcal{A}\in\mathbb{C}_{t}^{m, n}$, $\mathcal{B}\in\mathbb{C}_{t}^{m, n}$ and let $\mathcal{B}^{(1)_k}\in\mathbb{C}_{k}^{n, m}$ is an
order $k$ $\{1\}$ inverse of $\mathcal{B}\in\mathbb{C}$.
If $\mathcal{A}=P\mathcal{B}Q$, then $\mathcal{A}^{(1)_k}=Q^{-1}\mathcal{B}^{(1)_k}P^\mathrm{T}$, where $P\in\mathbb{C}^{m\times m}$ is a permutation matrix, $Q\in\mathbb{C}^{n\times n}$ is an invertible matrix.
\end{cor}

Obviously, when $P, Q$ are both permutation matrices, Theorem \ref{THM 2} also holds.

In the following, we show the definitions of the \textit{\{i\} inverse} and group inverse of tensors, $k$-T-idempotent tensors and idempotent tensors.
\begin{defi}\label{defi [3]}
Let $ \mathcal{A}\in\mathbb{C}_{t}^{m,n}$ and $ \mathcal{X}\in\mathbb{C}_{k}^{n,m}$.
\begin{eqnarray*}
{\rm{(1)}}&~&\mathcal{A} \mathcal{X} \cdot [(\mathcal{A}\cdot{y}^{t-1})^{[\frac{1}{s}]}]^s= \mathcal{A}\cdot{y}^{t-1},~for~all~y\in\mathbb{C}^n;\\
{ \rm{(2)}}&~&\mathcal{X} \mathcal{A} \mathcal{X}\cdot (y^{[\frac{1}{s}]})^{s(k-1)}= \mathcal{X}\cdot y^{k-1},~for~all~y\in\mathbb{C}^m;\\
 {\rm{(5)}}&~&\mathcal{A} \mathcal{X}\cdot (y^{[\frac{1}{s}]})^s= \mathcal{X}\cdot [(\mathcal{A}\cdot{y}^{t-1})^{[\frac{1}{s}]}]^{k-1},~for~all~y\in\mathbb{C}^n;
\end{eqnarray*}
where $s=(t-1)(k-1)$. If the equation ${\rm(i) }$ holds, then $ \mathcal{X}$ is called the order $k$ ${\rm\{i\} }$ inverse of $\mathcal{A}$, denoted by $\mathcal{X}=\mathcal{A}^{(i)_k}$. And the set of all the order $k$ {\rm \{i\}} inverses of $\mathcal{A}$ is denoted by $\mathcal{A}\{i\}_k$. For a tensor $ \mathcal{A}\in\mathbb{C}_{t}^{n,n}$, if the equations $(1)$, $(2)$ and $(5)$ hold for all $y\in\mathbb{C}^n$, then the tensor $\mathcal{X}\in\mathbb{C}_{k}^{n,n}$ is called an order $k$ group inverse of $\mathcal{A}$, denoted by $\mathcal{A}^{\#_k}$. And the set of all the order $k$ group inverse of $ \mathcal{A}$ is denoted by $ \mathcal{A}\{\#\}_k$.
\end{defi}
If an order $k$ right inverse of $\mathcal{A}\in\mathbb{C}_{t}^{n,n}$ exists, then $\mathcal{AA}^{R_k} =\mathcal{I}$, $\mathcal{A}^{R_k}\mathcal{AA}^{R_k}\cdot (y^{[\frac{1}{s}]})^{s(k-1)}=\mathcal{A}^{R_k}\mathcal{I}\cdot (y^{[\frac{1}{s}]})^{s(k-1)}=\mathcal{A}^{R_k}\cdot y^{k-1}$, where $s=(t-1)(k-1)$ and $y\in\mathbb{C}^n$ is an arbitrary vector. Hence, $\mathcal{A}^{R_k}$ is a \{2\} inverse of $\mathcal{A}$. It is clear that the \{i\} inverse of a tensor $\mathcal{A}\in \mathbb{C}_t^{m,n}$ is not unique in general and the group inverse of a tensor $\mathcal{A}\in \mathbb{C}_t^{n,n}$ ($t\geq 3$) is not unique in general. When $t=k=2$, the Definition \ref{defi [3]} is the definitions of the \{i\} inverse and the group inverse of matrices (see \cite{GI}).

Similar to Proposition \ref{unit 1th}, we can obtain the following result.
\begin{pro}
Let $\mathcal{A}={\rm diag}(a_1,a_2,\ldots,a_n)\in\mathbb{C}_{t}^{n,n}$ is a diagonal tensor.
Then
$${\rm diag}((a_1^+)^{\frac{1}{t-1}},(a_2^+)^{\frac{1}{t-1}}, \ldots,(a_n^+)^{\frac{1}{t-1}})$$
is an order $k$ group inverse of $\mathcal{A}$.
\end{pro}
Similar to the \{1\} inverse of a unit tensor, the group inverse (with fixed order) of a unit tensor is not unique in general.
\begin{thm}
Let $\mathcal{A}$ be the diagonal block tensor as in {\rm(\ref{block 4})}. Then ${\rm diag}(\mathcal{A}_1^{{\#}_k}, \mathcal{A}_2^{{\#}_k})\in\mathbb{C}_{k}^{n,n}$ is an order $k$ group inverse of $\mathcal{A}$.
\end{thm}
\begin{proof}
Let $\mathcal{G}_1=\mathcal{A}_1^{{\#}_k}$, $\mathcal{G}_2=\mathcal{A}_2^{{\#}_k}$ and  $\mathcal{G}={\rm diag}(\mathcal{G}_1, \mathcal{G}_2)$. And $y=\left(
                                                                                                           \begin{array}{c}
                                                                                                             Y_1 \\
                                                                                                             Y_2 \\
                                                                                                           \end{array}
                                                                                                         \right)
\in\mathbb{C}^n$ is an arbitrary vector, where $Y_1=(y_1, \ldots, y_r)^\mathrm{T}$ and $Y_2=(y_{r+1}, \ldots, y_n)^\mathrm{T}$.

It follows from Theorem \ref{diag 1} that $\mathcal{G}$ is an order $k$ {$\{1\}$ inverse} of $\mathcal{A}$.

Similar to Theorem \ref{diag 1}, we have the vector $\mathcal{G}\cdot y^{k-1}=\left(
                                                                     \begin{array}{c}
                                                                       {\mathcal{G}_1\cdot Y_1^{k-1}} \\
                                                                       {\mathcal{G}_2\cdot Y_2^{k-1}}\\
                                                                     \end{array}
                                                                   \right)$
 and $\mathcal{GA}={\rm diag}(\mathcal{G}_1\mathcal{A}_1, \mathcal{G}_2\mathcal{A}_2)$.
Let $z_1=\mathcal{G}_1\cdot (Y_1^{[\frac{1}{s}]})^{k-1}$, $z_2=\mathcal{G}_2\cdot (Y_2^{[\frac{1}{s}]})^{k-1}$ and $z =\left(
          \begin{array}{c}
            {z_1} \\
            {z_2} \\
          \end{array}
        \right)
=\mathcal{G}\cdot (y^{[\frac{1}{s}]})^{k-1}$, where $s=(t-1)(k-1)$.
By calculating, it yields that
\[
(\mathcal{GA}\cdot z^s)_i = \left\{ {\begin{array}{*{20}{c}}
   (\mathcal{G}_1\mathcal{A}_1\mathcal{G}_1\cdot (Y_1^{[\frac{1}{s}]})^{s(k-1)})_i=(\mathcal{G}_1\cdot Y_1^{k-1})_i,~if~i\leq r;  \\
   (\mathcal{G}_2\mathcal{A}_2\mathcal{G}_2\cdot (Y_2^{[\frac{1}{s}]})^{s(k-1)})_i=(\mathcal{G}_2\cdot Y_2^{k-1})_i,~if~i>r.  \\
\end{array}} \right.
\]
From the above discussion, we have $\mathcal{GA}\cdot z^s=\left(
                                 \begin{array}{c}
                                   {\mathcal{G}_1\cdot Y_1^{k-1}}\\
                                   {\mathcal{G}_2\cdot Y_2^{k-1}}\\
                                 \end{array}
                               \right)
=\mathcal{G}\cdot y^{k-1}$, so $\mathcal{G} \mathcal{A} \mathcal{G}\cdot (y^{[\frac{1}{s}]})^{s(k-1)}= \mathcal{G}\cdot y^{k-1}$.
Hence, we have $\mathcal{G}$ is an order $k$ {\{2\} inverse } of $\mathcal{A}$.

By the general tensor product, it yields that $$\mathcal{G}\cdot [(\mathcal{A}\cdot y^{t-1})^{[\frac{1}{s}]}]^{k-1}=\left(
                                                 \begin{array}{c}
                                                   {\mathcal{G}_1\cdot [(\mathcal{A}_1\cdot Y_1^{t-1})^{[\frac{1}{s}]}]^{k-1}}\\
                                                   {\mathcal{G}_2\cdot [(\mathcal{A}_2\cdot Y_2^{t-1})^{[\frac{1}{s}]}]^{k-1}}\\
                                                 \end{array}
                                               \right)
                                               $$
and
$$﹛
\mathcal{A}\mathcal{G}\cdot (y^{[\frac{1}{s}]})^s=\left(
                                                 \begin{array}{c}
                                                   {\mathcal{A}_1\mathcal{G}_1\cdot (Y_1^{[\frac{1}{s}]})^{s}}\\
                                                   {\mathcal{A}_2\mathcal{G}_2\cdot (Y_2^{[\frac{1}{s}]})^s}\\
                                                 \end{array}
                                               \right).$$
According to the definition of the tensor $\{5\}$ inverse. we get $\mathcal{G}\cdot [(\mathcal{A}\cdot y^{t-1})^{[\frac{1}{s}]}]^{k-1}=\mathcal{A}\mathcal{G}\cdot (y^{[\frac{1}{s}]})^s$.
Hence, we get $\mathcal{G}$ is an order $k$ {\{5\} inverse} of $\mathcal{A}$.

Thus, we get $\mathcal{G}$ is an order $k$ {group inverse} of $\mathcal{A}$.
\end{proof}
\begin{thm}\label{t 3}
Let $\mathcal{A}\in\mathbb{C}_{t}^{m,n}$, $\mathcal{B}\in\mathbb{C}_{t}^{m,n_1}$ and let $\mathcal{B}^{(2)_k}\in\mathbb{C}_{k}^{n_1,m}$ be an order $k$ $\{2\}$ inverse of $\mathcal{B}$.
If $\mathcal{A}=P\mathcal{B}Q$, then $\mathcal{A}^{(2)_k}=Q^{(2)}\mathcal{B}^{(2)_k}P^\mathrm{T}$, where $P\in\mathbb{C}^{m\times m}$ be a permutation matrix,  $Q\in\mathbb{C}^{n_1\times n}$ be a matrix with full row rank.
\end{thm}
\begin{proof}
Since $Q$ is a full row rank matrix, then $QQ^{(2)}=I$, where $I$ is a unit matrix (see \cite{GI}).
Let $\mathcal{G}=Q^{(2)}\mathcal{B}^{(2)_k}P^\mathrm{T}$. By computation, it yields that
\begin{eqnarray*}
 \mathcal{GA}\mathcal{G}\cdot (y^{[\frac{1}{s}]})^{s(k-1)}&=&Q^{(2)}\mathcal{B}^{(2)_k}P^\mathrm{T}P\mathcal{B}QQ^{(2)}\mathcal{B}^{(2)_k}P^\mathrm{T}\cdot (y^{[\frac{1}{s}]})^{s(k-1)}\\
&=&Q^{(2)}\mathcal{B}^{(2)_k}\mathcal{B}\mathcal{B}^{(2)_k}\cdot [(P^\mathrm{T}y)^{[\frac{1}{s}]}]^{s(k-1)},
\end{eqnarray*}
 where $s=(t-1)(k-1)$ and $y\in\mathbb{C}^m$ is an arbitrary vector. It follows from the definition of tensor {\{2\} inverse} that $$Q^{(2)}\mathcal{B}^{(2)_k}\mathcal{B}\mathcal{B}^{(2)_k}\cdot [(P^\mathrm{T}y)^{[\frac{1}{s}]}]^{s(k-1)}=Q^{(2)}\mathcal{B}^{(2)_k}P^\mathrm{T}\cdot y^{k-1}=\mathcal{G}\cdot y^{k-1}.$$
Thus, we get $ \mathcal{GA}\mathcal{G}\cdot (y^{[\frac{1}{s}]})^{s(k-1)}=\mathcal{G}\cdot y^{k-1}$, so $\mathcal{G}$ is an order $k$ {\{2\} inverse} of $\mathcal{A}$.

\end{proof}

\begin{thm}
Let $A\in\mathbb{C}^{n\times n}$ and $\mathcal{G}\in\mathbb{C}_{k}^{n,n}$ be an order $k$ group inverse of $A$. If $\lambda$ is an eigenvalue of $A$, then ${\lambda} ^+$ is an eigenvalue of $\mathcal{G}$.
\end{thm}
\begin{proof}
Let $\lambda$ be an eigenvalue of $A$, then $Ax=\lambda x$, $0 \neq x\in \mathbb{C}^n$. Clearly, $(Ax)^{[\frac{1}{k-1}]}=\lambda^{\frac{1}{k-1}}x^{[\frac{1}{k-1}]}$.
By the (1) in Definition \ref{defi [3]}, it yields that $A\mathcal{G}\cdot [(Ax)^{[\frac{1}{k-1}]}]^{k-1}=Ax=\lambda x$.
Since $A\mathcal{G}\cdot (\lambda^{\frac{1}{k-1}}x^{[\frac{1}{k-1}]})^{k-1}=\lambda A\mathcal{G}\cdot (x^{[\frac{1}{k-1}]})^{k-1}$, it is obtained that
\begin{eqnarray}\label{equ 1}
\lambda A\mathcal{G}\cdot (x^{[\frac{1}{k-1}]})^{k-1}=\lambda x.
\end{eqnarray}
It follows from the (5) in Definition \ref{defi [3]} that
\begin{align}
A\mathcal{G}\cdot (x^{[\frac{1}{k-1}]})^{k-1}=&\mathcal{G}\cdot [(Ax)^{[\frac{1}{k-1}]}]^{k-1}\notag \\
=&\mathcal{G}\cdot (\lambda^{\frac{1}{k-1}}x^{[\frac{1}{k-1}]})^{k-1}\notag \\
=&\lambda \mathcal{G}\cdot (x^{[\frac{1}{k-1}]})^{k-1}.\notag
\end{align}
That is
\begin{eqnarray}\label{equ 2}
A\mathcal{G}\cdot (x^{[\frac{1}{k-1}]})^{k-1}=\lambda \mathcal{G}\cdot (x^{[\frac{1}{k-1}]})^{k-1}.
\end{eqnarray}
Applying Eq.(\ref{equ 1}) and (\ref{equ 2}), it yields that
$$\lambda^2\mathcal{G}\cdot (x^{[\frac{1}{k-1}]})^{k-1}=\lambda x.$$

If $\lambda\neq 0$, then $\mathcal{G}\cdot (x^{[\frac{1}{k-1}]})^{k-1}=\lambda^{-1} (x^{[\frac{1}{k-1}]})^{[k-1]}$. Hence, $\lambda^{-1}$ is an eigenvalue of $\mathcal{G}$.

If $\lambda= 0$, then there exists a vector $0\neq x\in\mathbb{C}^n$ such that ${A}x=0$. Substituting it into the (5) of Definition \ref{defi [3]}, it yields that $A\mathcal{G}\cdot (x^{[\frac{1}{k-1}]})^{k-1}=\mathcal{G}\cdot [(Ax)^{[\frac{1}{k-1}]}]^{k-1}=0$. And it follows from the (2) of Definition \ref{defi [3]} that $\mathcal{G}\cdot x^{k-1}=\mathcal{G}A\mathcal{G}\cdot (x^{[\frac{1}{k-1}]})^{(k-1)^2}=0$. Hence, $0$ is an eigenvalue of $\mathcal{G}$.
\end{proof}
From the above theorem, it is easy to see that if $\lambda$ is an eigenvalue of a matrix $A$, then $\lambda^+$ is an eigenvalue of the matrix $A^{\#}$.

Let $\mathbb{H}\subset \mathbb{C}^n$ and $\mathbb{H}^{[s]}=\left\{ x^{[s]}|x\in \mathbb{H} \right\}$, where $s\geq 0$.
\begin{pro}
Let $\mathcal{A}\in \mathbb{C}_t^{m,n}$, let $\mathcal{A}^{(i)_k}\in \mathbb{C}_k^{n,m}$ and $\mathcal{A}^{\#_k}\in \mathbb{C}_k^{n,n}$ be the order $k$ \{i\} inverse and group inverse of $\mathcal{A}$, respectively.
Then the following results hold:

{\rm (1)} $\textbf{R}(\mathcal{AA}^{(1)_k})=\textbf{R}(\mathcal{A})$;

{\rm (2)} $\textbf{R}(\mathcal{A}^{(2)_k}\mathcal{A})=\textbf{R}(\mathcal{A}^{(2)_k})$;

{\rm (3)} $\textbf{R}(\mathcal{AA}^{\#_k})=\textbf{R}(\mathcal{A})$, $\textbf{R}(\mathcal{A}^{\#_k}\mathcal{A})=\textbf{R}(\mathcal{A}^{\#_k})$;

{\rm (4)} $\textbf{R}(\mathcal{A})\subset \textbf{R}(\mathcal{A}^{\#_k})$;

{\rm (5)} $\textbf{N}(\mathcal{A})\subset \left(\textbf{N}(\mathcal{AA}^{(5)_k})\right)^{[s]}$;

{\rm (6)} $\left(\textbf{N}(\mathcal{AA}^{(2)_k})\right)^{[s]}\subset \textbf{N}(\mathcal{A}^{(2)_k})$;

{\rm (7)} $\textbf{N}(\mathcal{A})\subset \textbf{N}(\mathcal{A}^{(2,5)_k})$,
{\rm (} $\mathcal{A}^{(2,5)_k}\in \mathbb{C}_k^{n,n}$ is an order $k$ \{2\} inverse and \{5\} inverse of $\mathcal{A}$ {\rm )};\\
where $s=(t-1)(k-1)$.

\end{pro}
\begin{proof}
{\rm (1)} It is easy to see that $\textbf{R}(\mathcal{AA}^{(1)_k})\subset \textbf{R}(\mathcal{A})$. For each $y\in \textbf{R}(\mathcal{A})$, there exists a vector $x\in \mathbb{C}^n$ such that
$y=\mathcal{A}\cdot x^{t-1}$. Since $y=\mathcal{A}\cdot x^{t-1}=\mathcal{AA}^{(1)_k}\cdot [(\mathcal{A}\cdot x^{t-1})^{[\frac{1}{s}]}]^s\in \textbf{R}(\mathcal{AA}^{(1)_k})$, we have $\textbf{R}(\mathcal{A})\subset \textbf{R}(\mathcal{AA}^{(1)_k})$.

{\rm (2)} It is clear that $\textbf{R}(\mathcal{A}^{(2)_k}\mathcal{A})\subset \textbf{R}(\mathcal{A}^{(2)_k})$. For each $y\in \textbf{R}(\mathcal{A}^{(2)_k})$, there exists $x\in \mathbb{C}^m$
such that $y=\mathcal{A}^{(2)_k}\cdot x^{k-1}$. Since $y=\mathcal{A}^{(2)_k}\mathcal{AA}^{(2)_k}\cdot (x^{[\frac{1}{s}]})^{s(k-1)}\in \textbf{R}(\mathcal{A}^{(2)_k}\mathcal{A})$,
we get $\textbf{R}(\mathcal{A}^{(2)_k})\subset \textbf{R}(\mathcal{A}^{(2)_k}\mathcal{A})$.

{\rm (3)} From the above results (1) and (2), it is easy to see that (3) holds.

{\rm (4)} For each $y\in \textbf{R}(\mathcal{A})$, it follows from (3) that there exist a vector $x\in \mathbb{C}^n$ such that .
Since $y=\mathcal{AA}^{\#_k}\cdot[(x^{[s]})^{[\frac{1}{s}]}]^s=\mathcal{A}^{\#_k}\cdot [(\mathcal{A}\cdot (x^{[s]})^{t-1})^{[\frac{1}{s}]}]^{k-1}$, it yields that $y\in \textbf{R}(\mathcal{A}^{\#_k})$. That is $\textbf{R}(\mathcal{A})\subset \textbf{R}(\mathcal{A}^{\#_k})$.

(5) For each $x\in \textbf{N}(\mathcal{A})$, we have $\mathcal{A}\cdot x^{t-1}=0$, so
$$\mathcal{AA}^{(5)_k}\cdot (x^{[\frac{1}{s}]})^s=\mathcal{A}^{(5)_k}\cdot[(\mathcal{A}\cdot x^{t-1})^{[\frac{1}{s}]}]^{k-1}=0.$$
Hence, $x^{[\frac{1}{s}]}\in\textbf{ N}(\mathcal{AA}^{(5)_k})$. That is $x\in \left(\textbf{ N}(\mathcal{AA}^{(5)_k})\right)^{[s]}$.

(6) For each $x\in \left(\textbf{N}(\mathcal{AA}^{(2)_k})\right)^{[s]}$, we have $x^{[\frac{1}{s}]}\in \textbf{N}(\mathcal{AA}^{(2)_k})$, so $\mathcal{AA}^{(2)_k}\cdot (x^{[\frac{1}{s}]})^s=0$.
Multiplying by $\mathcal{A}^{(2)_k}$ on the left hand side of the above equation, it yields that
$$\mathcal{A}^{(2)_k}\mathcal{AA}^{(2)_k}\cdot (x^{[\frac{1}{s}]})^{s(k-1)}=\mathcal{A}^{(2)_k}\cdot x^{k-1}=0.$$
Hence, $x\in \textbf{N}(\mathcal{A}^{(2)_k})$.

(7) For each $x\in \textbf{N}(\mathcal{A})$, there exists a vector $x\in \mathbb{C}^n$ such that $\mathcal{A}\cdot x^{t-1}=0$. So $(\mathcal{A}\cdot x^{t-1})^{[\frac{1}{s}]}=0$. It follows from Definition \ref{defi [3]} that
\begin{align}
\mathcal{A}^{(2,5)_k}\cdot x^{k-1}&=\mathcal{A}^{(2,5)_k}\mathcal{A}\mathcal{A}^{(2,5)_k}\cdot (x^{[\frac{1}{s}]})^{s(k-1)}\notag \\
&=\mathcal{A}^{(2,5)_k}\mathcal{A}^{(2,5)_k}\cdot[(\mathcal{A}\cdot x^{t-1})^{[\frac{1}{s}]}]^{(k-1)^2} \notag \\
&=0.\notag
\end{align}
 Hence, $x\in \textbf{N}(\mathcal{A}^{(2,5)_k})$.
\end{proof}

\section{$k$-T-idempotent tensors and idempotent tensors}

In this section, we give the definitions of the $k$-T-idempotent tensors and idempotent tensors first.

\begin{defi}\label{defi [2]}
Let $ \mathcal{A}\in\mathbb{C}_{t}^{n,n}$ and the positive integer $k\geq 2$.  If the equation
$$\mathcal{A}^k\cdot (y^{[\frac{1}{s}]})^{(t-1)^k}= \mathcal{A}\cdot y^{t-1}$$
 holds for all $y\in\mathbb{C}^n$,
where $s=(t-1)^{(k-1)}$, then $ \mathcal{A}$ is called the $k$-T-idempotent tensor. When $k=2$, $ \mathcal{A}$ is called the T-idempotent tensor.
\end{defi}

If $\mathcal{A}\in\mathbb{C}_{t}^{n,n}$ is a T-idempotent tensor, we have $\mathcal{A}^2\cdot (y^{[\frac{1}{t-1}]})^{(t-1)^2}= \mathcal{A}\cdot y^{t-1}$ for all $y\in\mathbb{C}^n$.
By computing, it yields that
\begin{eqnarray*}
\mathcal {A}^k \cdot(y^{[\frac{1}{{(t - 1)}^{k-1}}]})^{(t-1)^k}  &=& \mathcal {A}^{k - 2} \mathcal {A}^2\cdot[ (y^{[\frac{1}{{(t - 1)^{k - 2} }}]})^{[\frac{1}{{t - 1}}]}]^{(t-1)^k}\\
 &=& \mathcal {A}^{k - 2} \mathcal {A}\cdot (y^{[\frac{1}{{(t - 1)^{k - 2} }}]} )^{s}\\
 &=& \mathcal {A}^{k - 1} \cdot (y^{[\frac{1}{{(t - 1)^{k - 2} }}]} )^{s}\\
 & =&  \cdots  = \mathcal {A}\cdot y^{t-1},
\end{eqnarray*}
where $s=(t-1)^{(k-1)}$. Hence, a T-idempotent tensor is a $k$-T-idempotent tensor for all the positive integer $k\geq 2$.

\begin{defi}
Let $ \mathcal{A}\in\mathbb{C}_{t}^{n,n}$, if the equation $\mathcal{A}\cdot[(\mathcal{A}\cdot y^{t-1})^{[\frac{1}{t-1}]}]^{t-1}= \mathcal{A}\cdot y^{t-1}$ holds for all $y\in\mathbb{C}^n$.
Then $ \mathcal{A}$ is called the idempotent tensor.
\end{defi}
\begin{pro}
Let $\mathcal{A}\in\mathbb{C}_{t}^{n,n}$.

{\rm (1)} If $\mathcal{A}$ is a T-idempotent tensor, then $\mathcal{A}$ is an order $t$ {\rm \{2\}} inverse of itself;

{\rm (2)} If $ \mathcal{A}$ is both idempotent tensor and T-idempotent tensor, then $\mathcal{A}$ is an order $t$ {\rm \{1\}} inverse of itself and $\mathcal{A}\cdot (b^{[\frac{1}{(t-1)^2}]})^{t-1}$ is a solution of the solvable equation $\mathcal{A}\cdot x^{t-1}=b$, where $x\in\mathbb{C}^n$.
\end{pro}
\begin{proof}
{\rm (1)} Since $\mathcal{A}$ is a T-idempotent tensor, it yields that $\mathcal{A}^2\cdot (y^{[\frac{1}{t-1}]})^{(t-1)^2}= \mathcal{A}\cdot y^{t-1}$, for all $y\in \mathbb{C}^n$. Multiplying by $\mathcal{A}$ on the left hand side of the above equation, we get
$\mathcal{A}^3\cdot (y^{[\frac{1}{t-1}]})^{(t-1)^3}= \mathcal{A}^2\cdot y^{(t-1)^2}$, then $\mathcal{A}^3\cdot (y^{[\frac{1}{t-1}]})^{(t-1)^3}= \mathcal{A}\cdot (y^{[t-1]})^{t-1}$. Let $z=y^{[{t-1}]}$, then
$\mathcal{A}^3\cdot (z^{[\frac{1}{(t-1)^2}]})^{(t-1)^3}= \mathcal{A}\cdot z^{t-1}$. Thus, we have $\mathcal{A}=\mathcal{A}^{(2)_t}$.

{\rm (2)} Since $\mathcal{A}$ is a T-idempotent tensor, then $\mathcal{A}^2\cdot (z^{[\frac{1}{(t-1)}]})^{(t-1)^2}=\mathcal{A}\cdot z^{t-1}$, for all $z\in\mathbb{C}^n$. Let $z=(\mathcal{A}\cdot y^{t-1})^{[\frac{1}{(t-1)}]}$, $y\in\mathbb{C}^n$ is an arbitrary vector,
then
$$\mathcal{A}^2\cdot[(\mathcal{A}\cdot y^{t-1})^{[\frac{1}{(t-1)^2}]}]^{(t-1)^2}=\mathcal{A}\cdot[(\mathcal{A}\cdot y^{t-1})^{[\frac{1}{(t-1)}]}]^{t-1}.$$
Since $\mathcal{A}$ is an idempotent tensor, it yields that
$\mathcal{A}\cdot[(\mathcal{A}\cdot y^{t-1})^{[\frac{1}{(t-1)}]}]^{t-1}=\mathcal{A}\cdot y^{t-1}$. So $\mathcal{A}^2\cdot[(\mathcal{A}\cdot y^{t-1})^{[\frac{1}{(t-1)^2}]}]^{(t-1)^2}=\mathcal{A}\cdot y^{t-1}$. Hence,  $\mathcal{A}$ is an order $t$ {$\{1\}$ inverse} of itself. From Proposition \ref{pro 1}, it is easy to see that $\mathcal{A}\cdot(b^{[\frac{1}{(t-1)^2}]})^{t-1}$ is a solution of the solvable equation $\mathcal{A}\cdot x^{t-1}=b$.
\end{proof}

\begin{thm}
If $\mathcal{A}\in\mathbb{C}_{t}^{n,n}$ is a $k$-T-idempotent tensor, then the eigenvalues of $\mathcal{A}$ are the roots of $\lambda^{(t-1)^{k}}=1$ or $0$.
\end{thm}
\begin{proof}
Let $\lambda$ be an eigenvalue of $\mathcal{A}$, then $\mathcal{A}\cdot x^{t-1}=\lambda x^{[t-1]}$, $0\neq x\in\mathbb{C}^n$.
Multiplying by $\mathcal{A}^{k}$ on the left hand side of it, we get
\begin{eqnarray*}
\mathcal{A}^{k+1}\cdot x^{(t-1)^{k+1}} &=& A^{k}\cdot( \lambda x^{[t - 1]} )^{(t-1)^k} = \lambda ^{(t - 1)^{k} } \mathcal{A}^{k}\cdot( x^{[t - 1]})^{(t-1)^k}\\
 &=& \lambda ^{(t - 1)^{k} } \mathcal{A}^{k }\cdot\left[ \left( {\left( {x^{[t - 1]} } \right)^{\left[ {s } \right]} } \right)^{[\frac{1}{{s }}]}\right]^{(t-1)^k},
\end{eqnarray*}
where $s=(t-1)^{(k-1)}$.
Since $\mathcal{A}$ is a $k$-T-idempotent tensor, it yields that $$\mathcal{A}^k\cdot(y^{[\frac{1}{{s}}]})^{(t-1)^k}  = \mathcal{A}\cdot y^{t-1}, $$ for all $y\in\mathbb{C}^n$. Then
\begin{eqnarray*}
\mathcal{A}^{k+1}\cdot x^{(t-1)^{k+1}}&=&\lambda ^{(t - 1)^{k} } \mathcal{A}^{k }\cdot\left[ \left( {\left( {x^{[t - 1]} } \right)^{\left[ {s} \right]} } \right)^{[\frac{1}{{s }}]}\right]^{(t-1)^k} \\
& =& \lambda ^{(t - 1)^{k } } \mathcal{A}\cdot\left[\left( {x^{[t - 1]} } \right)^{\left[ {s} \right]} \right]^{t-1}\\
 & =& \lambda ^{(t - 1)^{k } } \mathcal{A}\cdot (x^{[(t - 1)^{k }] } )^{t-1} .
\end{eqnarray*}
So, $\mathcal{A}^{k+1}\cdot x^{(t-1)^{k+1}} = \lambda ^{(t - 1)^{k } } \mathcal{A}\cdot(x^{[(t - 1)^{k }] } )^{t-1}$.
By the $k$-T-idempotence, we have
\[
\mathcal{A}^{k+1}\cdot x^{(t-1)^{k+1}}  = \mathcal{A}^{k+1}\cdot\left[ \left( {x^{\left[ {(t - 1)^{k } } \right]} } \right)^{[\frac{1}{{(t - 1)^{k } }}]} \right]^{(t-1)^{k+1}} = \mathcal{A}\cdot (x^{\left[ {(t - 1)^{k} } \right]})^{t-1}.
\]
Hence,
\[
 \lambda ^{(t - 1)^{k } } \mathcal{A}\cdot (x^{[(t - 1)^{k }] } )^{t-1}=\mathcal{A}\cdot (x^{\left[ {(t - 1)^{k} } \right]})^{t-1},
\]
that is \[(\lambda ^{(t - 1)^{k } }-1 )\mathcal{A}\cdot (x^{\left[ {(t - 1)^{k} } \right]})^{t-1}=0.\]
If $\mathcal{A}\cdot (x^{\left[ {(t - 1)^{k} } \right]})^{t-1}\neq 0$, then $\lambda ^{(t - 1)^{k } }=1$. If $\mathcal{A}\cdot (x^{\left[ {(t - 1)^{k} } \right]})^{t-1}=0$, then $0$ is an eigenvalue of $\mathcal{A}$.
\end{proof}
\begin{cor}
{\rm (1)}
If $\mathcal{A}\in\mathbb{C}_{t}^{n,n}$ is a T-idempotent tensor, then the eigenvalues of $\mathcal{A}$ are the roots of $\lambda ^{{(t-1)}^2}=1$ or $0$;

{\rm (2)} If ${A}\in\mathbb{C}^{n\times n}$ is a $k$-idempotent matrix, then the eigenvalues of ${A}$ are the roots of $\lambda^k=1$ or $0$.
\end{cor}

\begin{thm}\label{THM 1}
If $\mathcal{A}\in\mathbb{C}_{t}^{n,n}$ is an idempotent tensor, then
the eigenvalues of $\mathcal{A}$ are $1$ or $0$.
\end{thm}
\begin{proof}
 Let $\lambda$ be an eigenvalue of $\mathcal{A}$, then $\mathcal{A}\cdot x^{t-1}=\lambda x^{[t-1]}$, $0\neq x\in\mathbb{C}^n$. So $(\mathcal{A}\cdot x^{t-1})^{[\frac{1}{t-1}]}=\lambda^{\frac{1}{t-1}} x$.
Multiplying by $\mathcal{A}$ on the left hand side of it, we obtain $\mathcal{A}\cdot[(\mathcal{A}\cdot x^{t-1})^{[\frac{1}{t-1}]}]^{t-1}=\lambda\mathcal{A}\cdot x^{t-1}$.
Since $\mathcal{A}$ is an idempotent tensor, $\mathcal{A}\cdot x^{t-1}=\lambda\mathcal{A}\cdot x^{t-1}$, that is $(\lambda-1)\mathcal{A}\cdot x^{t-1}=0$.
If $\mathcal{A}\cdot x^{t-1}\neq 0$, then $\lambda =1$. If $\mathcal{A}\cdot x^{t-1}=0$, then $0$ is an eigenvalue of $\mathcal{A}$.
\end{proof}

\begin{pro}
Let $\mathcal{A}\in\mathbb{C}_{t}^{m,n}$ and let $\mathcal{A}^{(1)_k},~\mathcal{A}^{(2)_k}\in\mathbb{C}_{k}^{n,m}$ be the order $k$ {\rm \{1\}} inverse and {\rm \{2\}} inverse of $\mathcal{A}$, respectively. Then

{\rm{(1)}} $\mathcal{A}\mathcal{A}^{(2)_k}$ is a T-idempotent tensor;

{\rm{(2)}} $\mathcal{A}\mathcal{A}^{(1)_k}$ is an idempotent tensor and the eigenvalues of it are $1$ or $0$.
\end{pro}
\begin{proof}
{\rm{(1)}}  Form the definition of the tensor {\{2\} inverse}, it yields that
$${({\cal A}{{\cal A}^{{{(2)}_k}}})^2}\cdot({y^{[\frac{1}{s}]}})^{s^2} = {\cal A}({{\cal A}^{{{(2)}_k}}}{\cal A}{{\cal A}^{{{(2)}_k}}})\cdot({y^{[\frac{1}{s}]}})^{s^2} = {\cal A}{{\cal A}^{{{(2)}_k}}}\cdot y^s,$$ for all $y\in\mathbb{C}^n$, where $s=(t-1)(k-1)$.
Hence, $\mathcal{A}\mathcal{A}^{(2)_k}$ is a T-idempotent tensor.

{\rm{(2)}} By the definition of the tensor \{1\} inverse, we have $\mathcal{A}\mathcal{A}^{(1)_k}\cdot[(\mathcal{A}\cdot z^{t-1})^{[\frac{1}{s}]}]^s=\mathcal{A}\cdot z^{t-1}$, for all $z\in\mathbb{C}^n$, where $s=(t-1)(k-1)$. Let $z=\mathcal{A}^{(1)_k}\cdot y^{k-1}$, $y\in\mathbb{C}^n$ is an arbitrary vector. Then $\mathcal{A}\mathcal{A}^{(1)_k}\cdot[(\mathcal{A}\mathcal{A}^{(1)_k}\cdot y^s)^{[\frac{1}{s}]}]^s=\mathcal{A}\mathcal{A}^{(1)_k}\cdot y^s$.
It yields that $\mathcal{A}\mathcal{A}^{(1)_k}$ is an idempotent tensor. And it follows form Theorem \ref{THM 1} that the eigenvalues of it are $1$ or $0$.
\end{proof}

\section{Some examples}

Let the vector $\alpha_i\in\mathbb{C}^{n_i}$ ($i=1,\ldots,t$), the \textit{outer product} of $\alpha_1, \ldots, \alpha_t$, denoted by $\alpha_1\otimes\cdots\otimes \alpha_t$, is a tensor $\mathcal{A}=(a_{i_1\cdots i_t})\in\mathbb{C}^{n_1\times\cdots\times n_t}$ with entries $a_{i_1\cdots i_t}=(\alpha_1)_{i_1}\cdots(\alpha_t)_{i_t}$, where $(\alpha_i)_j$ is the $j$-th component of $\alpha_i$.
The tensor $\mathcal{A}=(a_{i_1\cdots i_t})\in\mathbb{C}^{n_1\times\cdots\times n_t}$ can be decomposed into the form as
\[
\mathcal{A}=\sum_{j\in {[r]}}\alpha_{1}^j\otimes\cdots\otimes \alpha_{t}^j,
\]
where $ \alpha_{i}^j\in\mathbb{C}^{n_i}$ ($ i\in[t], j\in[r]$) (see \cite{Lim H}).
Let the matrix $B_k\in\mathbb{C}^{m_k\times n_k}$ ($k\in[t]$) and let $(B_k)_{j,i}$ be the $(j, i)$-entry of $B_k$. By the \textit{Tucker's product}, we get a tensor  $\mathcal{A'} =(a'_{j_1\cdots j_t})\in\mathbb{C}^{m_1\times\cdots\times m_t}$ as follows (see\cite{Lim H})
 \[
\mathcal{A}'=(B_{1},\ldots,B_{t})\cdot\mathcal{A} =\sum_{j\in[r]}B_1\alpha_{1}^j\otimes\cdots\otimes B_t\alpha_{t}^j,
 \]
where
 \[
 a'_{j_1\cdots j_t}=\sum_{i_1,\ldots,i_t=1}^{n_1,\ldots,n_t}(B_1)_{j_1,i_1}\cdots (B_t)_{j_t,i_t}a_{i_1\cdots i_t}.
 \]

If a tensor $\mathcal{A}\in\mathbb{C}_{t}^{m,n}$ can be decomposed into the form as
\[
\mathcal{A}=\sum\limits_{{i} \in [r]} \lambda_i e_i\otimes\alpha_i\otimes\alpha_i\otimes\cdots\otimes\alpha_i,
\]
where the vectors $\alpha_1,\ldots, \alpha_r \in\mathbb{C}^{n}$ are linearly independent, $\lambda_i\in \mathbb{C}$. Let the matrix $A=(\alpha_1~\alpha_2~\cdots~\alpha_r)\in\mathbb{C}^{n\times r}$ and $B=(e_1~e_2~\cdots~e_r)\in\mathbb{C}^{m\times r}$, $e_i$ is the unit vector with $i$-th component being $1$, then
\[
\mathcal{A}= (B,A,\ldots,A)\cdot \mathcal{D}=B\mathcal{D}A^{\rm{T}},
\]
where $\mathcal{D}\in\mathbb{C}_{t}^{r,r}$ is a diagonal tensor with the diagonal entries $\lambda_1, \ldots, \lambda_r$ (see \cite{S}).
Let $\mathcal{D}^{(1)_k},~\mathcal{D}^{(2)_k}\in \mathbb{C}_k^{r,r}$ be the order $k$ \{1\} inverse and \{2\} inverse of $\mathcal{A}$,
respectively. Similar to Theorem \ref{THM 2} and Theorem \ref{t 3}, we get
\[
\mathcal{A}^{(1)_k}=(A^\mathrm{T})^{(1)}\mathcal{D}^{(1)_k}B^\mathrm{T}\in \mathbb{C}_k^{n,m},
\]
\[
\mathcal{A}^{(2)_k}=(A^\mathrm{T})^{(2)}\mathcal{D}^{(2)_k}B^\mathrm{T}\in \mathbb{C}_k^{n,m}.
\]

Next, two examples are showed by the above discussion.
Let $\mathcal{A}=(\mathcal{A}_1|\mathcal{A}_2|\cdots|\mathcal{A}_t)\in \mathbb{C}^{n_1\times\cdots \times n_t}$, where $\mathcal{A}_i=(a_{ii_2\cdots i_t})$, $i\in[n_1]$.
\begin{example}
Let $\mathcal{A}$ be a $3\times 3\times 3$ tensor as follows
\[
\mathcal{A }= \left( {\begin{array}{*{20}{c}}
   {\left. {\begin{array}{*{20}{c}}
   1 & 2 & 3~  \\
   2 & 4 & 6~  \\
   3 & 6 & 9~  \\
\end{array}} \right|} & {\left. {\begin{array}{*{20}{c}}
   {16} & 8 & 4~  \\
   8 & 4 & 2~  \\
   4 & 2 & 1~  \\
\end{array}} \right|} & {\begin{array}{*{20}{c}}
   4 & 6 & 8  \\
   6 & 9 & {12}  \\
   8 & {12} & {16}  \\
\end{array}}  \\
\end{array}} \right),
\]
then the below tensor $\mathcal{B}$ is both a $\{1\}$ inverse and a $\{2\}$ inverse of $\mathcal{A}$,
\[
\mathcal{B }= \left( {\begin{array}{*{20}{c}}
   {\left. {\begin{array}{*{20}{c}}
   5 & -5 & 0~  \\
   5 & 1 & 0~  \\
   0 & 0 & -4~  \\
\end{array}} \right|} & {\left. {\begin{array}{*{20}{c}}
   {-14} & 14 & 0~  \\
   14 & -2 & 0~  \\
   0 & 0 & 11~  \\
\end{array}} \right|} & {\begin{array}{*{20}{c}}
   8 & -8 & 0  \\
   8 & 1 & {0}  \\
   0 & {0} & {-6}  \\
\end{array}}  \\
\end{array}} \right).
\]
\end{example}




In the following, the examples of T-idempotent tensor and idempotent tensor are given.
\begin{example}~\\
The tensor $\mathcal{A}=(\mathcal{A}_1\mid\mathcal{A}_2)=
\left( {\left. {\begin{array}{*{20}{c}}
   {25} & { - 15{\rm{ }}}  \\
   { - 15} & 9  \\
\end{array}} \right|\begin{array}{*{20}{c}}
   {{\rm{ }}100} & { - 60}  \\
   { - 60} & {36}  \\
\end{array}} \right)$ is an idempotent tensor;\\
The tensor $\mathcal{B}=(\mathcal{B}_1\mid\mathcal{B}_2)=
\left( {\left. {\begin{array}{*{20}{c}}
   9 & { - 6{\rm{ }}}  \\
   { - 6} & 4  \\
\end{array}} \right|\begin{array}{*{20}{c}}
   {36} & { - 24}  \\
   { - 24} & {16}  \\
\end{array}} \right)
$ is a T-idempotent tensor.

\end{example}

\end{document}